\documentclass[reqno]{amsart}
\usepackage{amssymb, amsmath, amsthm, amscd, mathrsfs}

\usepackage{paralist}
\usepackage{cite}
\usepackage{bigints}
\usepackage{dsfont}
\usepackage{empheq}
\usepackage{amssymb}
\usepackage{cases}
\usepackage{enumitem}
\allowdisplaybreaks

\usepackage{caption}
\usepackage{booktabs}

\usepackage{float}
\usepackage[ruled,vlined,linesnumbered,noresetcount]{algorithm2e}

\usepackage{hyperref}

\theoremstyle{plain}
\newtheorem{theorem}{Theorem}[section]
\newtheorem{lemma}[theorem]{Lemma}

\newtheorem{proposition}[theorem]{Proposition}
\newtheorem{definition}[theorem]{Definition}
\newtheorem{remark}{Remark}

\numberwithin{equation}{section}

\title[Controllability of Nonlocal PDEs] 
      {Controllability of Nonlocal Partial Differential Equations}

\author[D. Cardona]{Duván Cardona$^{1,2}$}
\author[S. E. Chorfi]{Salah-Eddine Chorfi$^{3}$}
\author[F. Et-tahri]{Fouad Et-tahri$^{3}$}
\author[L. Maniar]{Lahcen Maniar$^{3,4}$
\NoCaseChange{%
\\\\
\normalfont\small
$^1$~Department of Mathematics, King Fahd University of Petroleum and Minerals, Dhahran, Kingdom of Saudi Arabia.\\
$^2$~Department of Mathematics, Pontificia Universidad Javeriana, Bogot\'a-Colombia.\\
\texttt{duvan.sanchez@kfupm.edu.sa}\\
$^3$~Cadi Ayyad University, UCA, Faculty of Sciences Semlalia, Laboratory of Mathematics, Modeling and Automatic Systems, B.P. 2390, \\Marrakesh, Morocco\\
\texttt{s.chorfi@uca.ac.ma, f.ettahri@uca.ac.ma, maniar@uca.ac.ma}\\
$^4$~The UM6P-Vanguard Center, University Mohammed VI Polytechnic, Benguerir, Morocco.
}
}

\makeatletter 
\@namedef{subjclassname@2020}{%
  \textup{2020} Mathematics Subject Classification}
\makeatother


\begin{document}
\begin{abstract}
This chapter provides a short survey of recent developments in the controllability of nonlocal partial differential equations. We first review controllability results for evolution equations involving nonlocal spatial operators. Particular attention is devoted to unique continuation properties, observability inequalities, and Carleman estimates in establishing controllability properties. We then discuss controllability issues for partial differential equations with memory terms, highlighting how memory effects influence reachable states, null controllability, and exact controllability. We review and compare various analytical techniques developed to overcome the difficulties arising from the nonlocal structure. We also provide a critical assessment of some of these methods. Throughout the chapter, we emphasize the interplay between controllability and observability, the adaptation of classical control methods to nonlocal settings, and the distinctive phenomena that arise in comparison with local equations. We then discuss geometric inequalities and their role in the controllability of diffusion models. The chapter concludes with a discussion of current challenges and several open problems, aiming to provide directions for future research.
\end{abstract}

\maketitle

\section{Introduction and motivation}

This survey provides an overview of recent developments in the controllability of nonlocal partial differential equations, spectral inequalities, and the geometric analysis of these problems. Indeed, these phenomena appear in the description of nonlocal spatial interactions, arising naturally in many physical, biological, and engineering models. Typical examples include diffusion processes with long-range interactions, population dynamics, neural networks, and materials science; see, e.g., \cite{Coo05, Du12, Sil00} and the references therein. In contrast with classical local PDEs, the evolution at a given point depends not only on nearby values of the state but also on values over a region of the spatial domain.

Let $\Omega \subset \mathbb{R}^d$ be a bounded domain with boundary of class $C^2$. For $T>0$, we set
$$Q:=\Omega \times(0, T) \quad \text{ and } \quad \Sigma:=\partial \Omega \times(0, T).$$
A prototypical linear nonlocal parabolic equation is given by
\begin{equation}\label{eq-nonlocal}
\begin{cases}
y_t-\Delta y+\displaystyle\int_\Omega K(x,\xi,t)y(\xi,t)\,d\xi
=\mathds 1_\omega u
& \text{in } Q,\\
y=0 & \text{on } \Sigma,\\
y(\cdot,0)=y_0 & \text{in } \Omega,
\end{cases}
\end{equation}
where $K$ is a given real kernel, $\omega \Subset \Omega$ is a nonempty open subset (the control region), and $u$ is the control input.

The controllability of nonlocal partial differential equations has received significant attention in recent years. In contrast to classical local models, nonlocal equations involve interactions that depend either on values of the state at distant spatial points or on its history. Such nonlocal effects naturally arise in numerous applications, including heat conduction with memory, viscoelasticity, population dynamics, and anomalous diffusion. The presence of nonlocal terms significantly alters the system's qualitative behavior and often requires the development of new analytical tools for control problems.

The presence of the integral term destroys the locality of the equation and introduces substantial difficulties in the analysis of
controllability properties. In particular, the adjoint system also contains nonlocal terms, making the derivation of observability
inequalities considerably more involved.

\subsection*{Brief literature review}
Among the well-known and powerful controllability approaches, we find the Lebeau--Robbiano--Zuazua method \cite{LR95,LebeauLebeau1998} and the Carleman estimate method \cite{FI96}, which provide observability inequalities and null controllability results for the heat equation. These methods have inspired many subsequent developments for nonlocal models.

The study of controllability for equations involving nonlocal spatial interactions was initiated by Fernández-Cara, L\"u and Zuazua \cite{FLZ16}, who established null controllability results for heat and wave equations with nonlocal space terms. Their work demonstrated that, under suitable assumptions on the nonlocal kernel, classical observability techniques can be adapted to nonlocal settings. For parabolic systems with kernel integrals satisfying suitable analyticity assumptions, Lissy and Zuazua \cite{LZ18} obtained internal controllability results through a combination of observability estimates and analyticity arguments. Nonlocal effects have also been studied in nonlinear equations; Micu and Takahashi \cite{MH18} proved local controllability to stationary trajectories for a Burgers equation with nonlocal viscosity. The semilinear case was investigated by Hern\'andez-Santamar\'ia and Le Balc'h \cite{HL21}, who established local null controllability for a class of nonlocal semilinear heat equations. More recently, Allal, Fragnelli, and Salhi \cite{AFS24} investigated null controllability for degenerate parabolic equations containing nonlocal space terms, extending controllability results to settings where both degeneracy and nonlocality are present.

Further developments were proposed by Biccari and Hernández-Santamaría \cite{BU19}, who considered the null controllability of a heat equation with an additive integral term using Carleman estimates. However, as we explain below, their argument contains a gap that prevents the proof from being completed.

A second major line of research concerns equations with memory. These models involve convolution terms depending on the past evolution of the state and arise naturally in heat conduction with memory and viscoelastic systems. Early contributions include the work of Barbu and Iannelli \cite{BI00}, who studied controllability issues for the heat equation with memory terms. However, unlike the classical heat equation, memory effects may prevent controllability properties from holding. In this direction, Ivanov and Pandolfi \cite{IP09} established a lack of controllability to rest for the heat equation with memory, while Guerrero and Imanuvilov \cite{GI13} obtained further negative results concerning null controllability. These works revealed fundamental obstructions caused by memory terms.

At the same time, several positive controllability results have been obtained under boundedness and analyticity assumptions. Zhou and Gao \cite{ZG14} studied approximate and null controllability for heat equations with memory, and Zhou and Zhang \cite{ZZ18} extended these investigations to degenerate parabolic equations. A relevant result was achieved by Chaves-Silva, Zhang and Zuazua \cite{CZE17}, who introduced a novel approach to the controllability of evolution equations with memory based on observability inequalities for suitable adjoint systems. This framework was further developed by L\"u, Zhang and Zuazua \cite{LZZ17}, who established null controllability results for wave equations with memory. Later, Allal and Fragnelli \cite{AF21} analyzed null controllability for degenerate parabolic equations with memory, combining Carleman estimates with techniques adapted to degenerate operators. The paper \cite{WZZ24} by Wang et al. characterizes sharp observability for heat equations with time-dependent analytic memory kernels through a new approach based on flow decomposition, time analyticity, and propagation of singularities.

Despite these significant advances, the controllability theory of nonlocal PDEs remains incomplete. Many fundamental questions remain open, including the treatment of general nonlocal kernels, the development of sharp observability inequalities, the controllability of nonlinear and fractional models, and the interaction between nonlocality, degeneracy, and memory effects. The purpose of this chapter is to provide a unified overview of the existing literature on controllability and observability for nonlocal PDEs, emphasizing the analytical methods employed, the main achievements obtained so far, their limitations, and the challenging open problems that continue to drive current research.

\subsection*{Outline of the chapter}
This chapter is divided into two main parts according to the nature of the nonlocality under consideration. Section~\ref{sec1} focuses on controllability problems for PDEs with nonlocal space terms, while Section~\ref{sec2} is devoted to equations with memory effects. In each case, we survey the principal controllability and observability results available in the literature, emphasizing the techniques used to obtain them, particularly Carleman estimates and related observability methods. The chapter concludes in Section~\ref{sec3} with a discussion of current challenges and several open problems that may guide future developments in the field.

\section{Controllability of PDEs with Nonlocal Space Terms}\label{sec1}

\subsection{Linear heat and wave equations with nonlocal space terms}

The first systematic controllability results for evolution equations with nonlocal spatial interactions were obtained by Fern\'andez-Cara, L\"u and Zuazua \cite{FLZ16}. They studied both heat and wave equations containing integral space terms and proved controllability results under boundedness and analyticity assumptions on the kernel. Their approach relied on Carleman estimates combined with compactness–uniqueness and analyticity-based unique continuation. It showed that nonlocal perturbations can be treated as lower-order terms when appropriate bounds on the kernel are available.

This paper established that many classical controllability results for local equations remain valid in the presence of nonlocal spatial
interactions, provided the kernel satisfies suitable regularity and boundedness assumptions.

\subsection{Heat equations with integral kernels}
Let us consider the controllability problem
\begin{equation}\label{eq-nonlocal1}
\begin{cases}
y_t-\Delta y+\displaystyle\int_\Omega K(x,\xi)y(\xi,t)\,d\xi
=\mathds 1_\omega u
& \text{in } Q,\\
y=0 & \text{on } \Sigma,\\
y(\cdot,0)=y_0 & \text{in } \Omega,
\end{cases}
\end{equation}
where $u$ is a control supported in the region $\omega \Subset \Omega$.

\begin{definition}
The system \eqref{eq-nonlocal1} is null controllable at time $T$ if for all $y_0\in L^2(\Omega),$ there exists a control $u\in L^2(Q)$ such that $y(\cdot,T)=0$ in $L^2(\Omega)$. 
\end{definition}

By duality (see, e.g., \cite{FI96}), the null controllability of \eqref{eq-nonlocal1} is equivalent to the observability inequality
\begin{equation}\label{obs}
\|\varphi(\cdot,0)\|_{L^2(\Omega)}^2 \le C \int_{\omega\times (0,T)} |\varphi|^2 dx dt \qquad \forall \varphi_T \in L^2(\Omega),
\end{equation}
where $\varphi$ is the solution of the adjoint problem
\begin{equation}\label{adj}
\begin{cases}
-\varphi_t-\Delta \varphi+\displaystyle\int_\Omega K(\xi,x)\varphi(\xi,t)\,d\xi
=0
& \text{in } Q,\\
\varphi=0 & \text{on } \Sigma,\\
\varphi(\cdot,T)=\varphi_T & \text{in } \Omega.
\end{cases}
\end{equation}

\subsection{First ideas on controllability}
Now, we present some first ideas by Fern\'andez-Cara, L\"u and Zuazua \cite{FLZ16} for the null controllability of nonlocal parabolic PDEs. The main idea relies on the duality approach and a compactness--uniqueness argument.

Consider the nonlocal heat equation \eqref{eq-nonlocal1}.
\begin{theorem}[see \cite{FLZ16}] 
Fix $T>0$ and assume the following conditions on the kernel $K \in L^2(\Omega\times\Omega)$
	\begin{enumerate}
		\item \begin{equation*}
		x\mapsto
		\int_\Omega K(\xi,x)f(\xi)\,d\xi
		\quad\text{is analytic for every }f\in L^2(\Omega),
	\end{equation*}
		\item
		\begin{equation*}
        \begin{cases}
		K(x,\xi)=\displaystyle\sum_{m,j\ge1}k_{mj}\phi_m(x)\phi_j(\xi),\\
		\|K\|_R^2:=\displaystyle\sum_{m\ge1}
		\left(
		\sum_{j\ge1}\lambda_j^{-1}|k_{mj}|^2
		\right)
		\lambda_m^{-1}
		e^{2R\sqrt{\lambda_m}}
		<\infty,
        \end{cases}
\end{equation*}
where $\{\phi_j\}_{j\ge1}$ denotes an orthonormal basis of eigenfunctions of the Dirichlet Laplacian with corresponding eigenvalues $\{\lambda_j\}_{j\ge1}$, and $R=R(\Omega,\omega,T)$ is a sufficiently large constant.
	\end{enumerate}
Then, the system \eqref{eq-nonlocal1} is null controllable at every $T>0$.
\end{theorem}

\begin{proof}[Sketch of the main ideas of the proof]
	The proof relies on the Hilbert Uniqueness Method (HUM), combined with a compactness--uniqueness argument adapted to nonlocal operators.
    
	By duality, the null controllability of \eqref{eq-nonlocal1} is equivalent to proving the observability inequality \eqref{obs}.
	\medskip
    
	\noindent\textbf{Step 1. Decomposition of the adjoint solution.}
	The adjoint solution is decomposed as
	\begin{equation*}
	\varphi=p+\zeta,
	\end{equation*}
	where $p$ solves the classical backward heat equation 
	\begin{equation}\label{eq:adj heat}
		\begin{cases}
			-p_t-\Delta p=0, & \text{in }  Q,\\[1ex]
			p=0, & \text{on } \Sigma,\\
			p(\cdot,T)=\varphi_T, & \text{in } \Omega,
		\end{cases} 
	\end{equation}
	and $\zeta$ accounts for the nonlocal perturbation, i.e.,
	\begin{equation}
		\begin{cases}
			-\zeta_t-\Delta\zeta
			=-\displaystyle\int_\Omega
			K(\xi,x)\zeta(\xi,t)\,d\xi
			-\displaystyle\int_\Omega
			K(\xi,x)p(\xi,t)\,d\xi, &\text{in } Q,
			\\
			\zeta=0, &\text{on } \Sigma,
			\\
			\zeta(\cdot,T)=0, &\text{in } \Omega.
		\end{cases}
		\label{eq:zeta}
	\end{equation}
	This splitting separates the local dynamics from the nonlocal effects.
	
\medskip
\noindent\textbf{Step 2. Observability for the local component.}

The homogeneous backward heat equation \eqref{eq:adj heat} satisfies the
classical observability inequality. More precisely, introducing the quantity
\begin{eqnarray*}
\|\varphi_T\|_{h}^{2}
:=
\int_{0}^{T}\!\!\int_{\omega}|p(x,t)|^{2}\,dx\,dt,
\end{eqnarray*}
it is well known that $\|\cdot\|_{h}$ defines a norm on $L^{2}(\Omega)$.
Moreover, the global Carleman estimate for the heat equation (e.g., \cite{FI96}) implies the
weighted estimate
\begin{equation}\label{eq:carleman_p}
	\int_{Q}
	e^{-\frac{2C_0}{T-t}}
	|p|^{2}\,dxdt
	\le
	C
	\|\varphi_T\|_{h}^{2},
\end{equation}
where $C_0>0$ depends only on the geometry of the domain $\Omega$.

On the other hand, the spectral inequality (\cite{LR95}) yields
\begin{equation}\label{eq:spectral_norm}
	\sum_{j\ge1}
	e^{-2R\sqrt{\lambda_j}}
	|(\varphi_T,\phi_j)|^{2}
	\le
	C
	\|\varphi_T\|_{h}^{2},
\end{equation}

The proof of the observability estimate
\eqref{obs} is therefore reduced to establishing the following two inequalities:
\begin{align}
	\|\varphi_T\|_{h}^{2}
	&\le
	C
	\int_{0}^{T}\!\!\int_{\omega}
	|\varphi|^{2}\,dxdt,
	\label{eq:first_est}
	\\
	\|\varphi(\cdot,0)\|_{L^{2}(\Omega)}^{2}
	&\le
	C
	\|\varphi_T\|_{h}^{2}.
	\label{eq:second_est}
\end{align}
Indeed, combining \eqref{eq:first_est} and \eqref{eq:second_est} immediately gives \eqref{obs}.


\medskip
\noindent\textbf{Step 3. Proof of \eqref{eq:first_est}: compactness--uniqueness argument.}

The proof proceeds by contradiction. Assume that \eqref{eq:first_est} is false.
Then there exists a sequence
$\{\varphi_T^n\}_{n\ge1}\subset L^2(\Omega)$
such that
\begin{equation*}
	\|\varphi_T^n\|_{h}=1 \quad \text{and} \quad \int_{0}^{T}\!\!\int_{\omega}
	|\varphi_n|^{2}\,dxdt
	\rightarrow0,
\end{equation*}
where $\varphi_n$ denotes the solution of \eqref{adj} associated with the final datum $\varphi_T^n$.

Let $ \varphi_n=p_n+\zeta_n$, where $p_n$ solves the backward heat equation
\eqref{eq:adj heat} with final datum $\varphi_T^n$, and $\zeta_n$ satisfies \eqref{eq:zeta} for $p=p_n$.

The key point is to prove that the source term in \eqref{eq:zeta} is uniformly bounded in $L^2(0,T;H^{-1}(\Omega))$. Indeed, expanding the kernel in the eigenfunctions of the Dirichlet Laplacian and using the spectral representation of
$p_n$, together with assumptions on $K$ and estimate
\eqref{eq:spectral_norm},
one obtains
\begin{equation}
	\left\|
	\int_\Omega
	K(\xi,\cdot)p_n(\xi,\cdot)\,d\xi
	\right\|_{L^2(0,T;H^{-1}(\Omega))}
	\le C\|K\|_{R},
	\label{eq:forcing_bound}
\end{equation}
where the constant $C$ is independent of $n$.

Applying standard energy estimates to \eqref{eq:zeta}, we obtain
\begin{equation*}
\|\zeta_n\|_{L^2(0,T;H_0^1(\Omega))}
+
\|\partial_t\zeta_n\|_{L^2(0,T;H^{-1}(\Omega))}
\le C.
\end{equation*}
Hence, by the compact embedding $H^1_0(\Omega) \hookrightarrow L^2(\Omega)$ and the Aubin--Lions lemma, we can extract a subsequence such that
\begin{equation*}
\zeta_n
\rightarrow
\zeta
\qquad
\text{strongly in }
L^2(Q).
\end{equation*}
It remains to analyze the sequence
$\{\varphi_n\}$. Since
\begin{equation*}
\varphi_n=p_n+\zeta_n,
\end{equation*}
the estimate \eqref{eq:carleman_p}, together with the previous bounds, shows that $\{\varphi_n\}$ is uniformly bounded in
$L^2(\Omega\times(0,T-\delta))$ for every $\delta>0$. Therefore, up to another subsequence,
\begin{equation*}
\varphi_n
\rightharpoonup
\varphi
\qquad
\text{weakly in }
L^2(\Omega\times(0,T-\delta)).
\end{equation*}
Passing to the limit in the adjoint equation, we find that the limit function satisfies
\begin{equation*}
-\varphi_t-\Delta\varphi
+\int_\Omega
K(\xi,x)\varphi(\xi,t)\,d\xi=0
\quad\text{in }Q.
\end{equation*}
Moreover, the assumption
\begin{equation*}
\int_0^T\!\!\int_\omega |\varphi_n|^2
\rightarrow0
\end{equation*}
implies
\begin{equation*}
\varphi=0
\qquad
\text{in }\omega\times(0,T).
\end{equation*}
Since the mapping
\begin{eqnarray*}
x\mapsto
\int_\Omega
K(\xi,x)\varphi(\xi,t)\,d\xi
\end{eqnarray*}
is analytic and vanishes on the open set $\omega,$ it vanishes
everywhere in $\Omega$.
Consequently, $\varphi$ satisfies the classical backward heat equation and vanishes in the observation region $\omega$.

The unique continuation property for the heat equation then
gives $\varphi\equiv0$. Therefore, $\varphi_n\rightharpoonup0$ weakly in $L^2(\Omega\times(0,T-\delta))$. Using the boundedness of the integral operator, this implies that the source term in
\eqref{eq:zeta}
converges weakly to zero.
The energy estimates then yield
\begin{equation*}
\zeta_n
\rightarrow0
\qquad
\text{strongly in }
L^2(Q).
\end{equation*}
Finally,
since
\begin{equation*}
p_n=\varphi_n-\zeta_n,
\end{equation*}
we obtain
\begin{equation*}
\|p_n\|_{L^2(\omega\times(0,T))}
\longrightarrow0.
\end{equation*}
This contradicts the normalization
\begin{equation*}
\|p_n\|_{L^2(\omega\times(0,T))}
=
\|\varphi_T^n\|_{h}
=1.
\end{equation*}
Hence, estimate \eqref{eq:first_est} holds.


\medskip
\noindent\textbf{Step 4. Proof of \eqref{eq:second_est}.}

The estimate follows from the decomposition
$\varphi=p+\zeta$.

For the homogeneous part,
the monotonicity of the energy for the backward heat equation
gives
\begin{equation*}
\|p(\cdot,0)\|_{L^2(\Omega)}^{2}
\le
\|p(\cdot,t)\|_{L^2(\Omega)}^{2},
\qquad
0<t<T.
\end{equation*}

Integrating this inequality over a time interval contained in
$(0,T)$ and using the estimate
\eqref{eq:carleman_p},
one obtains
\begin{equation*}
\|p(\cdot,0)\|_{L^2(\Omega)}^{2}
\le
C
\|\varphi_T\|_{h}^{2}.
\end{equation*}
Next,
applying the energy estimate to equation
\eqref{eq:zeta}
yields
\begin{equation*}
\|\zeta(\cdot,0)\|_{L^2(\Omega)}^{2}
\le
C
\left\|
\int_\Omega
K(\xi,\cdot)p(\xi,\cdot)\,d\xi
\right\|_{L^2(0,T;H^{-1}(\Omega))}^{2}.
\end{equation*}
Using once again the spectral assumption on the kernel $K$ together
with estimate
\eqref{eq:spectral_norm},
we infer that
\begin{equation*}
\|\zeta(\cdot,0)\|_{L^2(\Omega)}^{2}
\le
C
\|\varphi_T\|_{h}^{2}.
\end{equation*}
Finally,
since
$\varphi=p+\zeta$,
the triangle inequality gives
\begin{equation*}
\|\varphi(\cdot,0)\|_{L^2(\Omega)}^{2}
\le
2\|p(\cdot,0)\|_{L^2(\Omega)}^{2}
+
2\|\zeta(\cdot,0)\|_{L^2(\Omega)}^{2}
\le
C
\|\varphi_T\|_{h}^{2},
\end{equation*}
which proves \eqref{eq:second_est}.

\end{proof}
\begin{remark}
	The main novelty of this approach is that it avoids establishing a Carleman estimate directly for the nonlocal equation. Instead, the nonlocal operator is viewed as a compact perturbation of the classical heat equation. The proof combines four key ingredients:
	\begin{enumerate}
		\item the classical observability inequality for the heat equation;
		\item the Lebeau--Robbiano spectral inequality;
		\item a compactness--uniqueness argument;
		\item the analyticity of the kernel, which restores the unique continuation property.
	\end{enumerate}
	This strategy has become a prototype for several subsequent controllability results involving lower-order nonlocal perturbations.
\end{remark}

Finally, we also mention \cite{FLNal19} for local exact controllability to the trajectories using a compactness–uniqueness argument and Kakutani’s fixed-point theorem.

\subsection{Carleman estimate method}
Recall that the standard method for deriving the observability inequality \eqref{obs} is based on a global Carleman estimate of the form
\begin{equation*}
\int_Q \rho^{-2}|\varphi|^2 d x d t \leq C \int_{\omega \times(0, T)} \rho^{-2}|\varphi|^2 d x d t,
\end{equation*}
where the constant $\mathcal{C}>0$ is independent of $\varphi_T$, and $\rho=\rho(x, t)$ is a suitable weight function that blows up as $t \to 0$ and $t \to T$. We refer to \cite{FI96} for more details. The Carleman estimate for parabolic equations yields
\begin{equation*}
\int_Q \rho^{-2}|\varphi|^2 d x d t \leq C \int_{\omega \times(0, T)} \rho^{-2}|\varphi|^2 d x d t  +C \int_{Q} \rho^{-2}\left|\int_{\Omega} K(\xi, x) \varphi(\xi, t) d\xi\right|^2 d x d t.
\end{equation*}
The last nonlocal term on the right-hand side cannot be absorbed by the left-hand side. This issue does not arise for classical potential terms acting locally in space, namely, for equations of the form
$$
-\varphi_t-\Delta \varphi+ p(x, t) \varphi(x, t)=0 .
$$

A key lemma to construct the Carleman weight function for parabolic equations reads as follows (see \cite[Lemma 1.1]{FI96}):
\begin{lemma}
Let $\omega \subset \subset \Omega$ be a nonempty open set. Then, there exists $\eta^0 \in C^2(\overline{\Omega})$ such that $\eta^0>0$ in $\Omega, \eta^0=0$ on $\partial \Omega$, and $\left|\nabla \eta^0\right|>0$ in $\overline{\Omega \backslash \omega}$.
\end{lemma}

For a parameter $\lambda>0$, we define
$$
\sigma(x):=e^{4 \lambda\left\|\eta^0\right\|_{\infty}}-e^{\lambda\left(2\left\|\eta^0\right\|_{\infty}+\eta^0(x)\right)},
$$
and we introduce the weight functions
\begin{equation}\label{2.1}
\alpha(x, t):=\frac{\sigma(x)}{t(T-t)}, \qquad \xi(x, t):=\frac{e^{\lambda\left(2\left\|\eta^0\right\|_{\infty}+\eta^0(x)\right)}}{t(T-t)}.
\end{equation}
Moreover, in what follows, we will use the notation
\begin{equation}
\begin{aligned}
\sigma^{+} & :=\max _{x \in \overline{\Omega}} \sigma(x)=e^{4 \lambda\left\|\eta^0\right\|_{\infty}}-e^{2 \lambda\left\|\eta^0\right\|_{\infty}}, \\
\sigma^{-} & :=\min _{x \in \overline{\Omega}} \sigma(x)=e^{4 \lambda\left\|\eta^0\right\|_{\infty}}-e^{3 \lambda\left\|\eta^0\right\|_{\infty}}. \label{2.2}
\end{aligned}
\end{equation}

An attempt was made by Biccari and Hern\'andez-Santamar\'ia \cite{BU19}, who considered the heat equation
with additive integral kernels $K=K(x, \xi, t) \in L^{\infty}(\Omega \times \Omega \times(0, T))$ such that
\begin{equation}
\mathcal{K}=: \sup_{(x, t) \in \overline{Q}} \exp \left(\frac{\sigma^{-}}{t(T-t)}\right) \int_{\Omega}|K(x, \xi, t)| d \xi<\infty, \tag{H}
\end{equation}
where $\sigma^{-}$ is the constant defined in \eqref{2.2} and depends on $\Omega$ and $\omega$.

The authors claimed null controllability results using Carleman estimates adapted to the nonlocal
structure under assumption $(\mathrm{H})$. However, their main result relies on the following proposition (see \cite[Proposition 2.3]{BU19}):
\begin{proposition}\label{prop1}
For any fixed $\lambda>0$ and $s>1$ it holds that
\begin{equation}\label{2.7}
\exp \left(-\frac{(1+s) \sigma^{-}}{t(T-t)}\right)<\exp \left(-\frac{s \sigma^{+}}{t(T-t)}\right),
\end{equation}
where $\sigma^{-}$ and $\sigma^{+}$ are defined in \eqref{2.2}.
\end{proposition}
This key proposition is incorrect. Indeed, the estimate \eqref{2.7} is equivalent to 
$$\frac{\sigma^{+}}{\sigma^{-}}<\frac{1+s}{s} \qquad \forall s>1.$$
For fixed $\lambda$, letting $s \to \infty$, we obtain $\sigma^{+}\leq \sigma^{-}$, which is a contradiction.

The error made in the proof of this proposition is contained in the last implication. The last implication incorrectly introduces the parameter $s$. Therefore, corrected assumptions can be given by invoking the parameter $s$ in $(\mathrm{H})$. However, since the Carleman parameter $s$ must be chosen sufficiently large later, such a condition is of limited practical usefulness. A suitable assumption on the kernel $K$ should reflect an intrinsic property of the model and yield a valid noncircular sufficient condition, rather than depend on the proof technique employed. We emphasize that the same criticism applies equally to similar works on systems with memory terms.

The study of semilinear models was subsequently addressed by Hern\'andez-Santamar\'ia and Le Balc'h \cite{HL21}, who established
local null controllability for a class of nonlocal semilinear heat equations. Their proof combined observability estimates for the
linearized problem with fixed-point arguments.

\subsection{Parabolic systems and analytic nonlocal terms}

Beyond scalar equations, controllability issues have also been studied for systems involving nonlocal couplings. Lissy and Zuazua
\cite{LZ18} investigated parabolic systems containing nonlocal kernels satisfying suitable analyticity assumptions and proved internal controllability results.
Their work highlighted the importance of analyticity assumptions in obtaining observability inequalities for coupled nonlocal systems.

These results illustrate that nonlocal interactions may preserve controllability properties even in complex multi-component systems, although the associated observability estimates become substantially more delicate.

\subsection{Nonlinear and degenerate equations}

Nonlocal effects frequently appear together with nonlinearities or degenerate diffusion mechanisms. Micu and Takahashi \cite{MH18}
studied a Burgers equation with nonlocal viscosity and proved local controllability to stationary trajectories. Their analysis combined
linearization techniques with inverse mapping arguments.

Subsequently, Allal, Fragnelli and Salhi \cite{AFS24} considered degenerate parabolic equations containing nonlocal space
terms. Using weighted Carleman estimates adapted to degenerate operators, they established null controllability results despite the
combined presence of degeneracy and nonlocality. This work represents one of the first contributions addressing these two difficulties
simultaneously.

\subsection{Controllability and geometric analysis of diffusion models for nonlocal elliptic operators}

This section describes the geometric approach developed in \cite{Cardona:Ruzhansky:Delgado:JGA2} by the first author together with J. Delgado and M. Ruzhansky for analyzing heat equations associated with higher-order elliptic differential operators and positive elliptic pseudo-differential operators.

To start with, we consider $(M,g)$ to be a compact $C^\infty$ Riemannian manifold. A fundamental result established by H. Donnelly and C. Fefferman in their influential paper \cite{DonnellyFefferman} shows that every eigenfunction of the Laplace operator $\Delta_g$ on $M$ satisfies a quantitative doubling estimate. More precisely, if
$$ 
-\Delta_g\phi=\lambda^2\phi,
$$
then there exist constants $C_1,C_2>0$, depending only on the manifold $M$ and independent of the radius $R>0$, such that
\begin{equation}\label{Fefferman-Donnelly}
\sup_{B(2R)}|\phi|\leq e^{C_1\lambda +C_2} \sup_{B(R)}|\phi|,
\end{equation}
where $B(R)$ and $B(2R)$ denote concentric geodesic balls. An important feature of this inequality is that it extends beyond individual eigenfunctions and continues to hold for finite sums of eigenfunctions of the Laplace--Beltrami operator.

A closely related problem concerns the geometry of nodal sets. In 1982, Yau proposed that if $M$ is an $n$-dimensional closed $C^\infty$ Riemannian manifold and $(\phi,\lambda^2)$ is an eigenpair of the Laplace operator, then the $(n-1)$-dimensional Hausdorff measure of the zero set should satisfy
$$
\lambda\lesssim H_{n-1}(\{x\in M:\phi(x)=0\})\lesssim \lambda.
$$
Here, $H_{n-1}$ denotes the $(n-1)$-dimensional Hausdorff measure. The doubling inequality \eqref{Fefferman-Donnelly} played a decisive role in this context: Donnelly and Fefferman used it to verify Yau's conjecture in the real-analytic setting; see \cite[Theorem 1.2]{DonnellyFefferman}.

This section reviews the results established in \cite{Cardona:Ruzhansky:Delgado:JGA2} concerning extensions of the doubling estimate \eqref{Fefferman-Donnelly} to finite sums of eigenfunctions of positive elliptic pseudo-differential operators $A$ on compact Lie groups. It also summarizes subsequent developments for positive pseudo-differential operators of arbitrary real order belonging to the global $(\rho,\delta)$-H\"ormander classes, together with recent advances in the more general setting of compact manifolds.

Before the results of \cite{Cardona:Ruzhansky:Delgado:JGA2}, it was unknown whether the doubling property could be established throughout the full range
$
0\leq \delta<\rho\leq 1,
$ 
for operators in the H\"ormander calculus. A related formulation of this problem had previously been brought to the attention of the first author by Zuazua. The compact Lie group framework adopted in \cite{Cardona:Ruzhansky:Delgado:JGA2} is particularly natural because, on arbitrary compact manifolds, the principal symbol of a pseudo-differential operator admits an invariant definition only under the restrictions $0\leq \delta<\rho\leq 1$ and $\rho\geq1-\delta$; see H\"ormander \cite{Hormander1985III}.

Rather than relying on Carleman estimates, as in the approach of Donnelly and Fefferman \cite{DonnellyFefferman}, the analysis developed in \cite{Cardona:Ruzhansky:Delgado:JGA2} follows a different strategy. This change in perspective is motivated by the lack of an appropriate Carleman theory for nonlocal operators on smooth manifolds. The central idea is to characterize those pseudo-differential operators for which a doubling inequality of the form \eqref{Fefferman-Donnelly} holds for finite sums of eigenfunctions by exploiting the representation-theoretic structure of compact Lie groups. To state the principal theorem proved in \cite{Cardona:Ruzhansky:Delgado:JGA2}, we first recall the notation and background needed in the sequel.

Following the strategy developed by Lebeau and Robbiano \cite{LR95}, controllability results for evolution equations associated with elliptic operators can be obtained from suitable spectral inequalities.

Let $E(x,D)\in \Psi^\nu_{1,0}(M)$ be a positive elliptic operator of order $\nu>0$ on a compact Riemannian manifold $(M,g)$, and consider the spectral data determined by
$$
E(x,D)\rho_j=\lambda_j^\nu\rho_j,\qquad
0\leq \lambda_j\leq\lambda_{j+1}\rightarrow\infty.
$$
In this direction, the following estimate was proved in \cite{Cardona} for positive elliptic differential operators on compact manifolds.
\begin{theorem}[\cite{Cardona}]\label{The:Spectral:Inequality}
Let $\nu\in 2\mathbb{N},$ and let $E(x,D)$ be a positive elliptic differential operator of order $\nu>0$ with smooth coefficients on $M.$ Let $(x,\xi)\in T^*M,$ and assume that for any $\xi\neq 0,$ the symbol satisfies $E(x,\xi)>0$. Then, for any non-empty open subset $\omega\subset M,$ we have
\begin{equation}\label{Spectral:Inequality:Intro2}
    \Vert \varkappa\Vert_{L^2(M)}\leq C_1e^{C_2 {\lambda}}\Vert \varkappa\Vert_{L^2(\omega)},\,\,\,\varkappa\in \textnormal{span}\{\rho_j:\lambda_j\leq \lambda\}.
\end{equation}
\end{theorem}

The proof of the spectral inequality in \cite{Cardona} relies on pseudo-differential techniques on compact manifolds without boundary, within the framework introduced by H\"ormander \cite{Hormander1985III}. This approach provides an alternative to the classical method of Donnelly and Fefferman \cite{DonnellyFefferman}, which is based on Carleman estimates. Such a modification is motivated by the fact that general pseudo-differential operators are nonlocal, and therefore the direct use of Carleman estimates is not available in this setting; see, for instance, the discussion in \cite[Subsection 1.4]{RousseauRobbiano2020}.

The argument in \cite{Cardona} combines tools from global pseudo-differential analysis on closed manifolds, including the Calder\'on--Vaillancourt theorem and the analytic functional calculus for elliptic operators. These techniques allow the derivation of an interpolation estimate of the form
\[
\Vert F\Vert_{H^1(M\times (\delta,T-\delta))}
\lesssim
\Vert F\Vert_{H^1(M\times (0,T))}^{\kappa}
\Vert F\Vert_{L^2(\omega)}^{1-\kappa},
\]
for $\kappa\in (0,1)$ and a suitable class of functions $F$, which plays a fundamental role in the proof of \eqref{Spectral:Inequality:Intro2}. A relevant difficulty in this framework is the possible nonlocality of the operators $E(x,D)$.

The result above admits a further extension to positive differential operators. In \cite{Cardona}, a Donnelly--Fefferman-type estimate was obtained for operators in the class $\Psi^\nu_{1,0}(M)$, thereby extending the classical doubling property of \cite{DonnellyFefferman}.

\begin{theorem}[\cite{Cardona}]\label{Donnelly-Type:Fefferman:Inequality}
Let $\nu>0,$ and let $E(x,D)\in \Psi^\nu_{1,0}(M)$ be an elliptic positive differential operator of order $\nu.$ Let $(x,\xi)\in T^*M,$ and assume that for any $\xi\neq 0,$ $E(x,\xi)>0$. 
For any $R>0$, let $B(x,R)$ be a ball defined by the geodesic distance, of radius $R>0$ and centered at $x.$ Then,
\begin{equation}\label{Donnelly-Fefferman2}
    \sup_{B(x,2R)}|\varkappa|\leq  e^{C_1' {\lambda}+C_2'} \sup_{B(x,R)}|\varkappa|,
    \,\,\,\varkappa\in \textnormal{span}\{\rho_j:\lambda_j\leq \lambda\},
\end{equation}
with $C_1'=C_{1}'(R)$ and $C_2'=C_2'(R)$ depending on the radius $R>0$ but not on $\varkappa.$
\end{theorem}

The spectral inequality above has direct applications to control theory. More precisely, using the Lebeau--Robbiano method \cite{LR95}, the estimate in Theorem \ref{The:Spectral:Inequality} yields null controllability results for fractional diffusion equations associated with elliptic operators on compact manifolds.

Consider the model problem
\begin{equation}\label{Main}
\begin{cases}
u_t(x,t)+ E(x,D)^\beta u(x,t)=g(x,t)\cdot \mathds{1}_\omega (x),& (x,t)\in M\times (0,T),\\
u(0,\cdot)=u_0,
\end{cases}
\end{equation}
where $T>0$. The problem is to determine the range of parameters $(\beta,\nu)$ for which, given any initial datum $u_0$, there exists a control function $g\in L^2(M\times(0,T))$ such that the corresponding solution satisfies
$$
u(x,T)=0,\qquad x\in M.
$$

The first null controllability result in this direction was obtained by Lebeau and Robbiano \cite{LR95}, who considered the heat equation generated by the Laplace--Beltrami operator $E(x,D)=-\Delta_g$ with $\beta=1$, including manifolds with smooth boundary. Later, Miller \cite{Miller2006} proved a general abstract result in separable Hilbert spaces showing that spectral inequalities imply null controllability for fractional powers $(-\Delta_g)^\beta$ whenever $\beta>1/2$. The spectral inequality for Laplace-Beltrami operators was also discussed in \cite{Lu2013}.

The extension of this framework to positive elliptic differential operators of arbitrary order was established in \cite{Cardona}. The following theorem, proved in \cite{Cardona}, gives the corresponding null controllability result for the fractional diffusion equation.

\begin{theorem}[\cite{Cardona}]\label{Main:theorem2}
Let $\nu\in 2\mathbb{N},$ and let $E(x,D)$ be a positive elliptic differential operator of order $\nu>0$ on $M$ with smooth coefficients. Let $(x,\xi)\in T^*M,$ and assume that for any $\xi\neq 0,$ $E(x,\xi)>0$. Then,
for any $\beta>1/\nu,$ the fractional diffusion problem in \eqref{Main} is null controllable at any time $T>0.$ 
\end{theorem}
\section{Controllability of PDEs with Memory Terms}\label{sec2}

Memory effects arise naturally in many applications where the current state of a system depends not only on its present configuration but
also on its history. Such physical phenomena appear in heat conduction with memory effects, viscoelasticity, diffusion processes in heterogeneous
media, and biological models. From a mathematical viewpoint, memory effects are commonly represented by convolution terms involving the
history of the state variable.

One of the first contributions to the controllability of equations with memory terms is due to Barbu and Iannelli \cite{BI00}, who studied the heat equation with memory and established controllability results under suitable assumptions on the memory kernel. Their work showed that the controllability properties of memory systems depend strongly on the kernel and cannot be inferred directly from those of the corresponding memoryless equation.

These pioneering results stimulated further investigations into the interaction between memory effects and control mechanisms, leading to
a deeper understanding of both positive and negative controllability phenomena.

A typical example is the heat equation with a memory term
\begin{equation}\label{eq-memory}
\begin{cases}
y_t-\Delta y+\displaystyle\int_0^t M(t-s)y(s)\,ds
=\mathds 1_\omega u
& \text{in } Q,\\
y=0 & \text{on } \Sigma,\\
y(\cdot,0)=y_0
& \text{in } \Omega,
\end{cases}
\end{equation}
where $M$ denotes the memory kernel and $u$ is the control acting on a subset $\omega \Subset \Omega$ of the spatial domain.

Although memory terms may appear as lower-order perturbations, their presence profoundly affects the controllability properties of the
system. Unlike equations with nonlocal space terms, memory effects introduce a dependence on the entire past trajectory, leading to new
phenomena that do not arise in classical evolution equations. Consequently, the controllability theory of PDEs with memory terms has
developed into a distinct and active area of research.

\subsection{Negative controllability results}

A remarkable feature of memory equations is that controllability may fail even when the corresponding classical equation is controllable. This phenomenon was first highlighted by Ivanov and Pandolfi \cite{IP09}, who proved that the heat equation with memory terms may fail to be null controllable to rest. Their analysis demonstrated that the memory term can create intrinsic obstructions preventing the state from being driven exactly to zero.

Further negative results were obtained by Guerrero and Imanuvilov \cite{GI13}. Through a detailed spectral analysis of the adjoint
system, they identified additional situations where null controllability does not hold. These works established that memory
effects fundamentally alter the controllability landscape and that classical parabolic intuition is often insufficient.

The existence of such counterexamples motivated the search for alternative notions of controllability and for conditions under which positive results can still be recovered.

\subsection{Further controllability results}

Despite the aforementioned difficulties, several positive and negative controllability results have been obtained. Zhou and Gao
\cite{ZG14} investigated controllability for heat equations with constant memory terms $M=a\in \mathbb{R}^*$. They proved that approximate controllability holds while null controllability fails. Their work demonstrated that approximate controllability can often be restored for exponential memory kernels.

The case of degenerate parabolic equations with memory was later considered by Zhou and Zhang \cite{ZZ18}. They extended the analysis to settings where both degeneracy and memory effects are present, showing that approximate controllability properties hold in an appropriate state space with a boundary control acting either on the degenerate or on the nondegenerate part of the boundary. However, null controllability fails again in this weakly degenerate case.

Recently, Allal and Fragnelli \cite{AF21} studied null controllability for degenerate parabolic equations with memory kernels $b(t,s,x)$ decaying near the final time. By combining Carleman estimates adapted to degenerate operators with careful estimates of the memory contribution, they obtained new controllability results extending the existing theory beyond uniformly parabolic frameworks.

\subsection{A unified approach to memory systems}

A systematic framework for the controllability of evolution equations with memory for particular kernel structures was developed by Chaves-Silva, Zhang and Zuazua \cite{CZE17}, using moving control supports. Rather than viewing memory terms as perturbations, they developed a technique for studying the controllability in the presence of memory terms. Their approach was based on the introduction of suitable observability inequalities for adjoint systems containing memory effects.

One of the key ideas was the distinction between classical null controllability and memory-type controllability, where both the state
and the accumulated memory are required to vanish at the final time. This perspective clarified several phenomena observed in earlier works and provided a unified methodology applicable to a broad class of evolution equations.

\subsection{Time-fractional PDEs}
Here, we discuss some controllability issues for time-fractional equations. We refer to \cite{LZ16} for the details.

Let $H$ be a Hilbert space and $A$ a linear unbounded operator on $H$, which generates a $C_0$-semigroup $\{S(t)\}_{t \geq 0}$ on $H$. Let $U$ be another Hilbert space and $B \in \mathcal{L}(U, H)$ be the control operator. For $\alpha \in (0,1)$, we consider the following control system:
\begin{equation}\label{tfeq}
\left\{\begin{array}{l}
\partial_{t}^\alpha y=A y + B u \qquad\text{ in }(0, \infty), \\
y(0)=y_0,
\end{array}\right.
\end{equation}
with an initial datum $y_0 \in H$ and a control $u \in L^2(0, T ; U)$. Here, the Caputo derivative $\partial_{t}^\alpha y$ is defined by
\begin{equation}\label{cd}
    \partial_t^{\alpha} y(t) :=\displaystyle \frac{1}{\Gamma(1-\alpha)} \int^t_0 (t-s)^{-\alpha} y^\prime(s) \, d s,
\end{equation}
where $\Gamma(\cdot)$ is the Gamma function.

Such a fractional model appears when modeling slow and anomalous diffusion phenomena in heterogeneous or porous media.

Recall that, in the classical sense, the system \eqref{tfeq} is said to be null controllable at time $T>0$ if for any $y_0 \in H$, there is a control $u \in L^2(0, T; U)$ such that the corresponding solution $y$ satisfies $y(T)=0$.

The above definition for time-fractional systems like \eqref{tfeq} gives only a ``partial" but not full null controllability. Indeed, due to memory effects arising from the integral term \eqref{cd}, the fact that the solution satisfies $y(T)=0$ does not guarantee that the solution stays at rest for $t\ge T$ when the control action stops. This issue does not arise in the integer-order case $\alpha=1$.

Therefore, a stronger notion of null controllability is needed for time-fractional systems.
\begin{definition}\label{defsc}
The system \eqref{tfeq} is strongly null controllable at time $T>0$ if for any $y_0 \in H$, there is a control $u \in L^2(0, T; U)$ such that the corresponding solution $y$ satisfies $y(t)=0$ for all $t\ge T$.
\end{definition}

Next, we present the lack of controllability of \eqref{tfeq} in the above stronger sense.
\begin{theorem}
The system \eqref{tfeq} is not strongly null controllable at any $T>0$ in the sense of Definition \ref{defsc}, for every bounded control operator $B$.
\end{theorem}

\begin{proof}
Let us assume that the system is strongly null controllable for $T>0$ in the sense of Definition \ref{defsc}. Let $y_0\neq 0$ and take $u(t)=0, \forall t \geq T$. Then, using the definition of the fractional derivative and the identity $y \equiv 0$ on $[T,\infty)$, we obtain
$$
 \int_0^t \frac{y^\prime(s)}{(t-s)^\alpha} \mathrm{d} s=Ay(t)+Bu(t)=0 \quad \forall t \geq T. \qquad
$$
This also yields
$$
\int_0^T \frac{y^\prime(s)}{(t-s)^\alpha} \mathrm{d} s=0 \quad \forall t \geq T .
$$
Thus, for any $\varphi \in D\left(A^*\right)$,
$$
 \int_0^T \frac{\langle y^\prime(s), \varphi\rangle_{D\left(A^*\right)^{\prime}, D\left(A^*\right)}}{(t-s)^\alpha} \mathrm{d} s=0, \quad \forall t>T .
$$
By differentiating the above identity under the integral sign with respect to $t$, we obtain
$$
\int_0^T \frac{\langle y^{\prime}(s), \varphi\rangle_{D\left(A^*\right)^{\prime}, D\left(A^*\right)}}{(t-s)^{\alpha+j}} \mathrm{~d} s=0 \quad \forall t>T, j \in\{0\} \cup \mathbb{N} .
$$
By the change of variables $\sigma=(t-s)^{-1}$, we have
$$
\int_{\frac{1}{t}}^{\frac{1}{t-T}}\left\langle y^{\prime}\left(t-\frac{1}{\sigma}\right), \varphi\right\rangle_{D\left(A^*\right)^{\prime}, D\left(A^*\right)} \sigma^{\alpha+j-2} \mathrm{~d} \sigma=0 \quad \forall t>T, \quad j \in\{0\} \cup \mathbb{N} .
$$
Setting
$$
f(\sigma)=\left\langle y^{\prime}\left(t-\frac{1}{\sigma}\right), \varphi\right\rangle_{D\left(A^*\right)^{\prime}, D\left(A^*\right)} \sigma^{\alpha-2},
$$
we then obtain
$$
\int_{\frac{1}{t}}^{\frac{1}{t-T}} f(\sigma) \sigma^j \mathrm{~d} \sigma=0 \quad \forall t>T, \quad j \in\{0\} \cup \mathbb{N} .
$$
The previous equality, together with the Weierstrass density theorem, implies that
$$
f(\cdot) \equiv 0 \quad \text { in }\left[\frac{1}{t}, \frac{1}{t-T}\right).
$$
Hence,
$$
y^{\prime} \equiv 0 \quad \text{ in }[0, T).
$$
Thus, $y(T)=0$ yields
$$
y \equiv 0 \quad \text{ in \,\, }[0, T].
$$
This contradicts $y_0\neq 0.$
\end{proof}
\bigskip

As for the observability inequality, let $z_0\in D\left(A^*\right)$ and consider the following adjoint system associated with \eqref{tfeq}:
\begin{equation}\label{tfadj}
\begin{aligned}
&\left\{\begin{array}{l}
\partial_{T^-}^\alpha z^\tau-A^* z^\tau=\dfrac{\alpha z_0}{\Gamma(1-\alpha)(\tau-t)^{1+\alpha}} \qquad\text { in }(0, T), \\
z^\tau(T)=0,
\end{array}\right.
\end{aligned}
\end{equation}
where the right-sided Caputo derivative $\partial_{T^-}^\alpha z$ is defined by
$$
\partial_{T^-}^{\alpha} z(t) :=-\frac{1}{\Gamma(1-\alpha)} \int_t^T (s-t)^{-\alpha} z^\prime(s) \, d s.
$$
Variants of this adjoint system were also considered in the literature.

Next, we state the duality between strong null controllability in the sense of Definition \ref{defsc} and a strong observability inequality.
\begin{proposition}
The system \eqref{tfeq} is strongly null controllable in time $T$ in the sense of Definition \ref{defsc} if and only if the adjoint system \eqref{tfadj} is initially observable, i.e., there is a constant $C>0$ such that for any $\tau \in(T,\infty)$ and $z_0 \in D\left(A^*\right)$, the following observability inequality holds
$$
\left\|\frac{1}{\Gamma(1-\alpha)} \int_0^T\left[\frac{z^\tau(t)}{t^\alpha}+\frac{z_0}{\tau^\alpha}\right] \mathrm{d} t\right\|_H \leq C\left\|B^* z^\tau\right\|_{L^2(0, T; U)}.
$$
\end{proposition}
In contrast to the integer-order case $\alpha=1$, we should observe not only the datum $z_0$ but also the history of the adjoint system solution $z$ in $[0, T]$, i.e., $\displaystyle\int_0^T \frac{z^\tau(t)}{t^\alpha} \mathrm{d} t$.

As for the classical controllability notion, the null controllability property has been studied in \cite{Xu17} for a time-fractional heat equation with the Caputo derivative of order $\alpha\in (0,1)$:
\begin{equation}\label{eq-frac}
\begin{cases}
\partial_t^\alpha y-\Delta y =\mathds 1_\omega u
& \text{in } Q,\\
y=0 & \text{on } \Sigma,\\
y(\cdot,0)=y_0 & \text{in } \Omega,
\end{cases}
\end{equation}
where $\omega$ is a measurable subset with positive measure.

The proof relies on duality and a spectral inequality to establish the observability inequality. However, there are some issues with this proof, which will be discussed in a forthcoming paper.

\section{Conclusion and open problems} \label{sec3}

In this chapter, we have surveyed recent developments in the controllability theory of partial differential equations involving nonlocal effects. Two major classes of nonlocal models have been considered: equations with nonlocal space terms and equations with memory. Although both types of nonlocality introduce long-range interactions, they lead to different mathematical challenges and require distinct analytical approaches.

For equations with nonlocal space terms, significant progress has been achieved over the last decade. Building upon classical observability techniques and Carleman estimates, several authors have established approximate and null controllability results for linear, semilinear, and degenerate equations. These works demonstrate that, under suitable assumptions on the nonlocal kernel, some controllability properties of local equations can be preserved. Nevertheless, the derivation of observability inequalities remains considerably more delicate than in the local setting due to the global nature of the spatial interactions.

The situation is markedly different for equations with memory. While positive controllability results have been obtained for several classes of models, memory effects may fundamentally alter the dynamics and even destroy classical controllability properties. The emergence of notions such as memory-type controllability has revealed the need for new perspectives and analytical tools that explicitly account for the influence of the history of the system. The development of observability inequalities adapted to memory equations has played a central role in this direction.

Despite the substantial advances reviewed in this chapter, the controllability of nonlocal PDEs remains an active area of research. Several important questions remain open and constitute promising directions for future research.

\begin{enumerate}
\item \textbf{Spatial integral terms with less restrictive kernel
assumptions.}
A natural direction is to determine which assumptions on the nonlocal kernel are essential for controllability and which can be
weakened. This includes kernels with limited regularity or singularities. The objective is to identify structural conditions
that yield observability estimates, together with counterexamples clarifying the limits of such extensions.

\item \textbf{Nonlinear nonlocal equations.}
An important direction is to identify conditions under which controllability of a linearized nonlocal equation extends to the
nonlinear system. Local controllability near a reference trajectory should be distinguished from global controllability. For
state-dependent kernels or nonlinearities beyond the globally Lipschitz setting, central issues include the existence of
controlled solutions on the prescribed time interval and observability estimates that are uniform over the relevant
family of linearized equations.

\item \textbf{Coupled systems and network models.}
For coupled PDEs and evolution equations on networks, a natural question is whether controls acting on selected components can
steer the entire system to the prescribed target. This requires relating controllability to the structure of the coupling
operators, the location of the controls, and, for network models, the graph geometry and transmission conditions. A particular
challenge is to formulate verifiable algebraic or spectral criteria for indirect control in the presence of nonlocal
couplings.

\item \textbf{Interaction between degeneracy, memory, and spatial
nonlocality.}
Another direction concerns equations in which degenerate diffusion interacts simultaneously with spatial integral terms
and memory effects. The formulation must specify the location and strength of the degeneracy, as well as whether the memory
acts on the state or on the diffusion operator. The objective is to distinguish the obstructions caused by degeneracy from those introduced by memory or spatial nonlocality.

\item \textbf{Memory-type controllability and controllability
to rest.}
Endpoint null controllability, simultaneous control of the state and selected memory variables, and controllability to rest
represent distinct objectives. For a general memory kernel, cancellation of the memory term at the terminal time need not
eliminate its subsequent contribution. A central question is therefore to characterize the kernel structures for which
suitable terminal constraints ensure that the system remains at rest after the control is switched off. For more general
kernels, one must determine which features of the history need to be controlled, taking into account the known
non-controllability results, particularly for time-fractional equations.

\item \textbf{Spectral inequalities for nonlocal elliptic operators
on manifolds with boundary.}
For specified positive self-adjoint elliptic realizations on compact manifolds with boundary, a natural objective is to
establish Lebeau--Robbiano-type spectral inequalities under additional structural assumptions ensuring quantitative unique
continuation. The operator domain and boundary conditions must be included in the formulation, since ellipticity and positivity
alone do not guarantee such inequalities for general nonlocal operators. Boundary value problems within the Boutet de Monvel
calculus or suitable extensions provide a framework for investigating this question. One possible strategy is to adapt
the interpolation approach of Lebeau and Zuazua \cite{LebeauLebeau1998}, while accounting for the additional
boundary contributions and nonlocal terms.
\end{enumerate}

Finally, the rapid growth of the literature over the past years illustrates both the mathematical richness and the practical relevance of nonlocal PDEs. As new models continue to emerge from applications, the development of controllability and observability theories capable of addressing increasingly complex nonlocal phenomena is expected to remain an active and fruitful area of research for years to come.

\end{document}